\documentclass{amsart}

\usepackage{graphicx}

\newtheorem{theorem}{Theorem}[section]
\newtheorem{lemma}[theorem]{Lemma}

\theoremstyle{definition}

\newtheorem{example}[theorem]{Example}

\theoremstyle{remark}
\newtheorem{remark}[theorem]{Remark}

\numberwithin{equation}{section}

\usepackage{chngcntr}
\counterwithout{equation}{section}
\counterwithout{theorem}{section}

\begin{document}

\title{Thin right-angled Coxeter groups in some uniform arithmetic lattices}

\author{Sami Douba}
\address{McGill University, Department of Mathematics and Statistics}
\email{sami.douba@mail.mcgill.ca}
\thanks{The author was supported by a public grant as part of the Investissement d'avenir project, FMJH, and by the National Science Centre, Poland UMO-2018/30/M/ST1/00668.}

%    General info

\begin{abstract}
 %apply a variant of an unpublished argument due to Agol to 
Using a variant of an unpublished argument due to Agol, we show that an irreducible right-angled Coxeter group on ${n \geq 3}$ vertices embeds as a thin subgroup of a uniform arithmetic lattice in an indefinite orthogonal group~$\mathrm{O}(p, q)$ for some $p, q \geq 1$ satisfying $p+q=n$. 
\end{abstract}

\maketitle

Let ${\bf G}$ be a semisimple algebraic $\mathbb{R}$-group and $\Gamma$ a lattice in $G := {\bf G}(\mathbb{R})$. A subgroup~$\Delta \subset \Gamma$ is said to be {\it thin} if $\Delta$ is Zariski-dense in $G$ but of infinite index in~$\Gamma$. It follows from the Borel density theorem \cite[Corollary~4.3]{MR123639} and a classical result of Tits \cite[Theorem~3]{MR286898} that if ${\bf G}$ as above is nontrivial, connected, and without compact factors, then any lattice in $G$ contains a thin nonabelian free subgroup. A famous construction of Kahn–Markovic \cite{MR2912704} produces thin surface subgroups of all uniform lattices in $\mathrm{SO}(3,1)$ (see \cite{MR3361773}, \cite{MR3665171}, \cite{MR3921320}, \cite{kahn2018surface} for some other manifestations of surface groups as thin groups). In \cite{MR4127090}, Ballas–Long show that many arithmetic lattices in $\mathrm{SO}(n,1)$ virtually embed as thin subgroups of lattices in $\mathrm{SL}_{n+1}(\mathbb{R})$, and raise the question as to which groups arise as thin groups. In this note, we observe the following.

\begin{theorem}\label{main} An irreducible right-angled Coxeter group on $n \geq 3$ vertices embeds as a thin subgroup of a uniform arithmetic lattice in $\mathrm{O}(p, q)$ for some $p, q \geq 1$ satisfying $p+q=n$. 
\end{theorem}

To that end, let $\Sigma_1$ be a connected simplicial graph on $n \geq 3$ vertices; we think of $\Sigma_1$ as a Coxeter diagram in the sense of \cite[Section~2.1]{felikson2014essential} all of whose edges are bold. Fix an order $v_1, \ldots, v_n$ on the vertices of $\Sigma_1$, and let $W$ be the group given by the presentation with generators $\gamma_1, \ldots, \gamma_n$ subject to the relations $\gamma_i^2 = 1$ for $i=1, \ldots, n$, and $[\gamma_i, \gamma_j]=1$ for each distinct $i,j \in \{1, \ldots n\}$ such that $v_i$ and $v_j$ are not adjacent in $\Sigma_1$. The group $W$ is the {\it (right-angled) Coxeter group} associated to the diagram $\Sigma_1$. Let $W^+$ be the index-$2$ subgroup of $W$ consisting of all elements that can be expressed as a product of an even number of the $\gamma_i$; that $W^+$ indeed constitutes an index-$2$ subgroup of $W$ follows, for instance, from faithfulness of the representation $\sigma_1$ of $W$ to be defined in the sequel.

For $d \in \mathbb{R}$, let $M_d = (m_{ij}) \in \mathrm{M}_n(\mathbb{Z}[d])$ be the symmetric matrix given by 
\[
m_{ij} = \begin{cases} 1 & \text{if } i= j \\ -d & \text{if }i \neq j \text{ and }v_i,v_j\text{ are joined by an edge in }\Sigma_1 \\ 0 & \text{otherwise.} \end{cases}
\]
Let $\epsilon > 0$ be such that $M_d$ is positive-definite for $d \in [-\epsilon, \epsilon]$, and let $D \geq 1$ be such that $M_d$ is nondegenerate and its signature constant as $d$ varies within $[D, \infty)$. 
Note that $M_1$ is the Gram matrix of the diagram $\Sigma_1$ (and the given order on the vertices of $\Sigma_1$). In particular, we have that $\epsilon < 1$. For $d > 1$, the matrix $M_d$ is the Gram matrix of the diagram $\Sigma_d$ obtained from $\Sigma_1$ by replacing each edge with a dotted edge labeled by $d$. (Here, we are again using the conventions employed by \cite[Section~2.1]{felikson2014essential}.)

For $d \geq 1$, let $\sigma_d : W \rightarrow \mathrm{GL}_n(\mathbb{R})$ be the Tits–Vinberg representation associated to the Coxeter diagram $\Sigma_d$ and the given order on its vertices; this is the representation given by
\[
\sigma_d(\gamma_i)(v) = v - 2 (v^T M_d e_i) e_i
\]
for $i = 1, \ldots, n$ and $v \in \mathbb{R}^n$, where $(e_1, \ldots, e_n)$ is the standard basis for $\mathbb{R}^n$. It follows from Vinberg's theory of reflection groups that the representations $\sigma_d$, $d \geq 1$, are faithful \cite[Theorem~5]{vinberg1971discrete} (see Lecture 1 in \cite{benoist2004five} for an exposition). This family of representations was studied in \cite{MR4139042}. 

If $M \in \mathrm{M}_n(\mathbb{R})$ is a symmetric matrix and $A$ is a subdomain of $\mathbb{C}$, we write
\begin{alignat*}{4}
\mathrm{O}(M; A) &= \{g \in \mathrm{GL}_n(A) \> : \> g^T M g = M \}, \\
\mathrm{SO}(M; A) &= \{g \in \mathrm{SL}_n(A) \> : \> g^T M g = M \}.
\end{alignat*}
Note that we have $W_d := \sigma_d(W) \subset \mathrm{O}(M_d; \mathbb{R})$ by design.

 %For $d > 1$, let $\Sigma_d$ be the Coxeter diagram obtained from $\Sigma_1$ by replacing each edge with a dotted edge labeled by $d$. Let $E$ be the free $\mathbb{R}$-vector space on $S$. For $d \geq 1$, let $B_d$ be the bilinear form on $E$ and $\sigma_d: W \rightarrow \mathrm{GL}(E)$ the Tits canonical representation of $W$ given by the Coxeter diagram $\Sigma_d$,  and let \[\mathrm{O}(E, B_d) = \{ g \in \mathrm{GL}(E) \> | \> B_d(gv, gw) = B_d(v,w) \> \text{for all } v,w \in E\}.\] By design, we have that $\sigma_d(W) \subset \mathrm{O}(E, B_d)$. Let $D \geq 1$ be such that $B_d$ is nondegenerate and its signature constant as $d$ varies within $[D, \infty)$. Note that for $d\geq 1$, the matrix $M_d$ is simply the Gram matrix of the diagram $\Sigma_d$ and the above order on the vertices of $\Sigma_d$. In particular, we have that $\epsilon < 1$. Note also that we obtain from our order on the vertices of $\Sigma_d$ an identification of $E$ with $\mathbb{R}^n$, and hence an identification of $\mathrm{GL}(E)$ with $\mathrm{GL}_n(\mathbb{R})$ under which the subgroup $\mathrm{O}(E, B_d)$ is identified with \[\mathrm{O}(M_d; \mathbb{R}) = \{ g \in \mathrm{GL}_n(\mathbb{R}) \> | \> g^T M_d g = M_d\}.\]

\begin{lemma}\label{Zariski}
The group $W_d$ is Zariski-dense in $\mathrm{O}(M_d; \mathbb{R})$ for $d \geq D$.
\end{lemma}

\begin{proof}
The proof of the main theorem in \cite{benoist2004adherence} applies here, so we only sketch the argument provided there. Let $d \geq D$ and let $G_d$ be the Zariski-closure of $W_d$ in~$\mathrm{O}(M_d; \mathbb{R})$.  Denote by $\mathfrak{g}$ and $\mathfrak{h}$ the Lie algebras of $\mathrm{O}(M_d; \mathbb{R})$ and $G_d$, respectively. It is enough to show that $\mathfrak{g} = \mathfrak{h}$, since the Zariski-closure of $\mathrm{SO}(M_d; \mathbb{R})^\circ$ is~$\mathrm{SO}(M_d; \mathbb{R})$ and since~$W_d \not\subset \mathrm{SO}(M_d; \mathbb{R})$. 

For each distinct pair $i,j \in \{1, \ldots, n\}$, let $E_{i,j}$ be the orthogonal complement of $\langle e_i, e_j \rangle$ in $\mathbb{R}^n$ with respect to $M_d$. The subgroup of $\mathrm{O}(M_d; \mathbb{R})$ consisting of all elements that fix each vector in $E_{i,j}$ is a $1$-dimensional closed subgroup of~$\mathrm{O}(M_d; \mathbb{R})$ whose identity component $T_{i,j}$ corresponds to a subspace $\langle X_{i,j} \rangle$ of $\mathfrak{g}$ for some~${X_{i,j} \in \mathfrak{g}}$. Since $M_d$ is nondegenerate, the elements $X_{i,j}$ form a basis for~$\mathfrak{g}$ as a vector space \cite[Lemme~7]{benoist2004adherence}. Thus, to show $\mathfrak{g} = \mathfrak{h}$, it suffices to show that~$X_{i,j} \in \mathfrak{h}$ for each distinct pair $i,j \in \{1, \ldots, n\}$.

To that end, let $i,j \in \{1, \ldots, n\}$, $i \neq j$, and suppose first that $v_i$ and $v_j$ are adjacent in $\Sigma_1$. Then $\sigma_d(\gamma_i \gamma_j)$ generates an infinite cyclic subgroup of $T_{i,j}$, so that~$T_{i,j} \subset G_d$. It follows that $X_{i,j} \in \mathfrak{h}$ in this case. One now verifies that, since $\Sigma_1$ is connected, any Lie subalgebra of $\mathfrak{g}$ that contains $X_{i,j}$ for all $i,j$ such that $v_i, v_j$ are adjacent in fact contains $X_{i,j}$ for each distinct pair $i,j \in \{1, \ldots, n\}$ \cite[Preuve~du~Th\'eor\`eme,~second~cas]{benoist2004adherence}.
\end{proof}

Now let $K \subset \mathbb{R}$ be a real quadratic extension of $\mathbb{Q}$, let $\tau: K \rightarrow K$ be the nontrivial element of $\mathrm{Gal}(K/\mathbb{Q})$, and let $\mathcal{O}_K$ be the ring of integers of $K$. Then by Dirichlet's unit theorem, there is a unit $\alpha \in \mathcal{O}_K^*$ such that $\alpha \geq \max\{\frac{1}{\epsilon}, D\}$. Thus, we have 
\[
\frac{|\tau(\alpha)|}{\epsilon} \leq \alpha |\tau(\alpha)| = | \alpha \cdot \tau (\alpha)| = 1,
\]
where the final equality holds because $\alpha \in \mathcal{O}_K^*$. We conclude that $|\tau(\alpha)| \leq \epsilon$, and so~$M_{\tau(\alpha)}$ is positive-definite. It follows that $\Gamma := \mathrm{O}(M_\alpha ; \mathcal{O}_K)$ is a uniform arithmetic lattice in~$\mathrm{O}(M_\alpha ; \mathbb{R})$ (for an efficient survey of the relevant facts, see, for instance, Section 2 of \cite{gromov1987non}). Moreover, we have $W_\alpha \subset \mathrm{O}(M_\alpha ; \mathbb{Z}[\alpha]) \subset \Gamma$. 

%Let $\Gamma = \mathrm{O}(M_\alpha ; \mathcal{O}_K)$. Then the image of $\Gamma$ in $\mathrm{O}(M_{\tau(\alpha)} ; \mathbb{R}) \times \mathrm{O}(M_\alpha ; \mathbb{R})$ under the map $ \tau \times \mathrm{Id}$ is a lattice. Since $\mathrm{O}(M_{\tau(\alpha)} ; \mathbb{R})$ is compact since $M_{\tau(\alpha)}$ is positive-definite, and so $\Gamma$ is in fact a uniform lattice in $\mathrm{O}(M_\alpha ; \mathbb{R})$. Moreover, we have $W_\alpha \subset \mathrm{O}(M_\alpha ; \mathbb{Z}[\alpha]) \subset \Gamma$. 

\begin{remark}
Note that Galois conjugation by $\tau$ transports $\Gamma$ and hence $W_\alpha$ into the compact group $\mathrm{O}(M_{\tau(\alpha)}; \mathbb{R})$. That right-angled Coxeter groups on finitely many vertices embed in compact Lie groups had already been observed by Agol \cite{agol2018hyperbolic} using a similar trick to the one above. Indeed, Agol's argument was the inspiration for this note.
\end{remark}

\begin{proof}[Proof of Theorem \ref{main}] We show that $W_\alpha$ is a thin subgroup of $\Gamma \subset \mathrm{O}(M_\alpha ; \mathbb{R})$. By Lemma \ref{Zariski}, it suffices to show that $W_\alpha$ is of infinite index in $\Gamma$. Indeed, suppose otherwise. Then $W_\alpha$ is a uniform lattice in $\mathrm{O}(M_\alpha ; \mathbb{R})$. If $n = 3$, then immediately we obtain a contradiction, since in this case $W_\alpha$ is virtually a closed hyperbolic surface group, whereas $W$ is virtually free. Now suppose $n > 3$. There is some $\beta > \alpha$ and a path $[\alpha, \beta] \rightarrow \mathrm{GL}_n(\mathbb{R}), d \mapsto h_d$ such that $h_d^T M_d h_d = M_\alpha$ for all $d \in [\alpha,\beta]$ (this follows, for example, from the fact that $\mathrm{GL}_n(\mathbb{R})$ acts continuously and transitively on the set $\Omega \subset M_n(\mathbb{R})$ of symmetric matrices with the same signature as $M_\alpha$, and so the orbit map $\mathrm{GL}_n(\mathbb{R}) \rightarrow \Omega, g \mapsto g^TM_\alpha g$ is a fiber bundle). Setting $g_d = h_dh_\alpha^{-1}$ for $d \in [\alpha, \beta]$, we have that $g_\alpha = I_n$ and $g_d^T M_d g_d = M_\alpha$ for $d \in [\alpha,\beta]$. For $d \in [\alpha,\beta]$, let $\rho_d = g_d^{-1} \sigma_d g_d$, and note $\rho_d(W) \subset g_d^{-1} \mathrm{O}(M_d ; \mathbb{R}) g_d = \mathrm{O}(g_d^TM_dg_d ; \mathbb{R}) = \mathrm{O}(M_\alpha ; \mathbb{R})$. 

Let $\rho_d^+ = \rho_d\bigr|_{W^+}$ and $\sigma_d^+ = \sigma_d\bigr|_{W^+}$ for $d \in [\alpha, \beta]$. Then $\rho_\alpha^+(W^+)$ is a uniform lattice in the connected non-compact simple Lie group $\mathrm{SO}(M_\alpha ; \mathbb{R})^\circ$, and the latter is not locally isomorphic to $\mathrm{SO}(2,1)^\circ$ by our assumption that $n>3$. Thus, by Weil local rigidity \cite{weil1960discrete,weil1962discrete}, up to choosing $\beta$ closer to $\alpha$, we may assume that for each $d \in [\alpha, \beta]$ there is some $a_d \in \mathrm{SO}(M_\alpha ; \mathbb{R})^\circ$ such that 
\begin{equation}\label{conj}
\rho_d^+ = a_d \rho_\alpha^+ a_d^{-1} = a_d \sigma_\alpha^+ a_d^{-1}.
\end{equation}
But $\rho_d^+ = g_d^{-1} \sigma_d^+ g_d$, so we obtain from (\ref{conj}) that the trace $\mathrm{tr}(\sigma_d (\gamma_i \gamma_j))$ remains constant as $d$ varies within $[\alpha, \beta]$, where $i,j \in \{1, \ldots, n\}$ are chosen so that the vertices $v_i, v_j$ are adjacent in $\Sigma_1$. 

We claim, however, that $\mathrm{tr}(\sigma_d (\gamma_i \gamma_j)) = 4d^2 - 4 +n$ for $d \geq D$. Indeed, let~$d \geq D$. Then $M_d$ is nondegenerate, so that $\mathbb{R}^d$ splits as a direct sum of the $2$-dimensional subspace $\langle e_i, e_j \rangle \subset \mathbb{R}^n$ and its orthogonal complement $E_{i,j}$ with respect to $M_d$. Each of $\gamma_i$ and $\gamma_j$ acts as the identity on $E_{i,j}$, so our claim is equivalent to the assertion that  $\mathrm{tr}\left(\sigma_d (\gamma_i \gamma_j)\bigr|_{\langle e_i, e_j \rangle}\right) = 4d^2-2$, and the latter follows from the fact that, with respect to the basis $(e_i, e_j)$ of ${\langle e_i, e_j \rangle}$, the matrices representing $\sigma_d(\gamma_i), \sigma_d(\gamma_j)$ are
$
\begin{pmatrix} -1 & 2d \\ 0 & 1 \end{pmatrix}, \begin{pmatrix} 1 & 0 \\ 2d & -1 \end{pmatrix},
$
respectively.
\end{proof}

\begin{example}
We consider the case that $n \geq 5$ and the complement graph of~$\Sigma_1$ is the cycle $v_1v_2\ldots v_n$. In this case, the group $W$ may be realized as the subgroup of $\mathrm{Isom}(\mathbb{H}^2)$ generated by the reflections in the sides of a regular right-angled hyperbolic $n$-gon, so that $W$ is virtually the fundamental group of a closed hyperbolic surface. We have
\begin{equation}\label{cycle}
M_d = (1+d)I_n + d(J_n + J_n^{n-1}) - d(I_n + J_n + \ldots + J_n^{n-1})
\end{equation}
where $J_n \in \mathrm{M}_n(\mathbb{C})$ is the matrix
\[
J_n = \begin{pmatrix} e_2 & e_3 & \ldots & e_n & e_1 \end{pmatrix}.
\]
There is some $C \in \mathrm{GL}_n(\mathbb{C})$ such that
\[
CJ_n C^{-1} = \mathrm{diag}(1, \zeta_n, \zeta_n^2, \ldots, \zeta_n^{n-1})
\]
where $\zeta_n = e^{2\pi i/n}$. Observe that
\begin{alignat*}{4}
C(I_n + J_n + \ldots + J_n^{n-1})C^{-1} &= \mathrm{diag}(n, 0, \ldots, 0) \\
C(J_n + J_n^{n-1})C^{-1} &= \mathrm{diag}\left(2, 2\cos \frac{2\pi }{n}, 2\cos \frac{2\pi \cdot 2}{n}, \ldots, 2 \cos \frac{2\pi(n-1)}{n}\right).
\end{alignat*}
It follows from (\ref{cycle}) that, counted with multiplicity, the eigenvalues of $M_d$ are ${1-d(n-3)}$ and $1 + d\left(1 + 2\cos\frac{2\pi k}{n} \right)$, where $k= 1, \ldots, n-1$. Note that for $d$ sufficiently large, we have that $1-d(n-3) < 0$, and that $1 + d\left(1 + 2\cos\frac{2\pi k}{n} \right) \geq 0$ if and only if $\cos\frac{2\pi k}{n} \geq -\frac{1}{2}$. We conclude that the signature of $M_d$ is $(2\lfloor \frac{n}{3} \rfloor, n-2\lfloor \frac{n}{3} \rfloor)$ for all~$d$ sufficiently large. In particular, if $n = 3m$, $m \geq 2$, then the signature of~$M_d$ is $(2m, m)$ for all $d$ sufficiently large. The above discussion yields thin surface subgroups of uniform arithmetic lattices in $\mathrm{SO}(2\lfloor \frac{n}{3} \rfloor, n-2\lfloor \frac{n}{3} \rfloor)$ for each $n \geq 5$.

\end{example}

\subsection*{Acknowledgements} I thank Yves Benoist and Pierre Pansu for helpful discussions. I am also deeply grateful to the latter for inviting me to spend the fall of 2021 at Universit\'e Paris-Saclay, where this note was written, and to my supervisor Piotr Przytycki for his support during my stay.
\bibliography{ThinRACGsbib}{}
\bibliographystyle{alpha}

\end{document}